\documentclass[12pt,reqno]{amsart}
\usepackage{latexsym,amssymb,amscd,tikz}
\usepackage{graphicx}
\def\kk{{\Bbbk}}
\def\B'c{{\mathcal{B'}}}
\def\U'c{{\mathcal{U'}}}

\def\xb{{\bold x}}
\def\mb{{\bold m}}

\def\opn#1#2{\def#1{\operatorname{#2}}} % to make operators
\opn\chara{char}
\opn\length{\ell}
\opn\projdim{proj\,dim}
\opn\cochord{cochord}
\opn\injdim{inj\,dim}
\opn\ini{in}
\opn\rank{rank}
\opn\height{ht}
\opn\Tiefe{Tiefe}
\opn\grade{grade}
\opn\height{ht}
\opn\embdim{emb\,dim}
\opn\codim{codim}

\opn\Tr{Tr}
\opn\bigrank{big\,rank}
\opn\superheight{superheight}\opn\lcm{lcm}
\opn\trdeg{tr\,deg}%
\opn\reg{reg}
\opn\lreg{lreg}
\opn\deg{deg}
\opn\lcm{lcm}
\opn\set{set}
\opn\ara{ara}

\opn\div{div}
\opn\Div{Div}
\opn\cl{cl}
\opn\Cl{Cl}
\opn\Spec{Spec}
\opn\Supp{Supp}
\opn\supp{supp}
\opn\Sing{Sing}
\opn\Ass{Ass}
\opn\Min{Min}

\opn\Ann{Ann}
\opn\Rad{Rad}
\opn\Soc{Soc}
\opn\Ker{Ker}
\opn\Coker{Coker}
\opn\Im{Im}
\opn\Hom{Hom}
\opn\Tor{Tor}
\opn\Ext{Ext}
\opn\End{End}
\opn\Aut{Aut}
\opn\id{id}

\opn\nat{nat}
\opn\GL{GL}
\opn\SL{SL}
\opn\mod{mod}
\opn\ord{ord}
\opn\depth{depth}
\opn\set{set}
\opn\Shad{Shad}
\opn\pd{pd}
\opn\indmat{indmat}
\opn\reg{reg}
\opn\dim{dim}
\opn\dist{dist}
\opn\aff{aff}
\opn\con{conv}
\opn\relint{relint}
\opn\st{st}
\opn\lk{lk}
\opn\cn{cn}
\opn\core{core}
\opn\vol{vol}
\opn\bight{bight}
\opn\gr{gr}

\opn\d{d}
\opn\Ind{Ind}
\opn\diam{diam}
\def\pot#1#2{#1[\kern-0.28ex[#2]\kern-0.28ex]}

\opn\dirlim{\underrightarrow{\lim}}
\opn\invlim{\underleftarrow{\lim}}
\def\pnt{{\raise0.5mm\hbox{\large\bf.}}}

\def\Implies{\ifmmode\Longrightarrow \else
     \unskip${}\Longrightarrow{}$\ignorespaces\fi}
\def\implies{\ifmmode\Rightarrow \else
     \unskip${}\Rightarrow{}$\ignorespaces\fi}
\def\iff{\ifmmode\Longleftrightarrow \else
     \unskip${}\Longleftrightarrow{}$\ignorespaces\fi}

\let\:=\colon
\newtheorem{Theorem}{Theorem}[section]
\newtheorem{Lemma}[Theorem]{Lemma}
\newtheorem{Corollary}[Theorem]{Corollary}
\newtheorem{Proposition}[Theorem]{Proposition}
\newtheorem{Remark}[Theorem]{Remark}

\newtheorem{Example}[Theorem]{Example}

\newtheorem{Problem}[Theorem]{Problem}

\newtheorem{Question}[Theorem]{Question}
\let\epsilon=\varepsilon
\let\phi=\varphi
\let\kappa=\varkappa
\title{Edge ideals of squares of trees}
\author{Anda Olteanu}
\address{Faculty of Marine Engineering, ``Mirceal cel B\u atr\^an" Naval Academy, Fulgerului Street, no. 1
	900218 Constanta, Romania,} \email{olteanuandageorgiana@gmail.com}
\begin{document}

\maketitle
\begin{abstract}  We describe all the trees with the property that the corresponding edge ideal of the square of the tree has a linear resolution. As a consequence, we give a complete characterization of those trees $T$ for which the square is co-chordal, that is the complement of the square, $(T^2)^c$, is a chordal graph. For particular classes of trees such as paths and double brooms we determine the Krull dimension and the projective dimension.
\end{abstract}

\section*{Introduction}
In 1960, Harary and Ross defined the squares of trees \cite{RH}, and their definition has been extended to squares of graphs. For a finite simple graph $G$, its square, denoted by $G^2$, is the graph with the same vertex set as $G$ and two vertices are adjacent in $G^2$ if they are adjacent in $G$ or their distance in $G$ is 2. Properties of  squares of graphs have been intensively studied in combinatorics \cite{AA,AMH,FH,FK,LT,LT1,L,R,R1,SW}. Classes of graphs which are closed to taking squares and more general taking powers, have been determined. Strongly chordal graphs \cite{L,R1}, interval graphs \cite{R}, proper interval graphs \cite{R} are known to have this property. Moreover, many researchers paid attention to recognition of squares of graphs and they developed algorithms for their recognition. The complexity of the problem of recognition of squares of graphs or of square roots of graphs has been determined for several classes of graphs.

Given a finite simple graph $G$ on the vertex set $V(G)=\{1,\ldots,n\}$ and with the set of edges $E(G)$, one may consider the edge ideal of $G$, denoted by $I(G)$, which is defined to be the squarefree monomial ideal in the polynomial ring $S=\kk[x_1,\ldots, x_n]$ which is generated by the squarefree monomials which correspond to the edges of $G$, that is $I(G)=\langle x_ix_j: \{i,j\}\in E(G)\rangle$. During the past decades, researchers described combinatorial properties of the graph $G$ in terms of algebraic and homological invariants of the graph $G$ and vice-versa \cite{F,HHZ,HT,HaTu1,HaTu,HeHi,II,Ka,K,MV,Vi,W}. A well-known example is the one of Fr\"oberg who proved that the edge ideal of a graph  has a linear resolution if and only if the graph is co-chordal, that is its complement is a chordal graph \cite{F}. Moreover, Woodroofe gave an upper-bound of the Castelnuovo--Mumford regularity in terms of the co-chordal number of a graph \cite{W}. The projective dimension, the Betti numbers and the Krull dimension of the edge ideal have also been related to combinatorial invariants of the graph (see for instance \cite{HaTu1,HaTu,K,MV,W}).

We aim at studying the behaviour of algebraic and homological invariants of the corresponding edge ideal of the square of a tree.  We give a complete characterization of the trees for which the edge ideal of their square has a linear resolution. 

The paper is structured as follows. In the first section we recall all the concepts that will be used through the paper. We distinguish here between combinatorial concepts arising from graph theory and notions from commutative algebra and describe the connections between them. The second section is devoted to recalling the characterization of squares of trees given by Harary and Ross \cite{RH}. We also prove a result which allows us to extend all the results obtained in this paper to larger classes of graphs, (Proposition~\ref{general}). 

In the third section, we characterize all trees for which the edge ideal of their square has a linear resolution. As a consequence, we give a complete characterization of those trees $T$ for which the square is co-chordal. 

The fourth section is devoted to particular classes of trees such as paths and double brooms. Since their square are chordal graphs, we may use the results developed by Kimura \cite{K}. By using combinatorial techniques, we compute invariants as the Krull dimension, the projective dimension and the Castelnuovo--Mumford regularity of edge ideals of squares of the path graph and of the double brooms.

In the end of the paper, we consider several problems that arise naturally. 

\section{Preliminaries}
\label{sec:1}
We review some standard facts on graph theory and edge ideals and we setup the notation and terminology that will be used through the paper. A more complete theory can be obtained by \cite{F,HaTu1,HeHi,MV,Vi}. 
\subsection{Notions from graph theory}
Let $G$ be a finite simple graph with the vertex set $V(G)$ and the set of edges $E(G)$. Two vertices $x,y\in V(G)$ are called \textit{adjacent} (or \textit{neighbours}) if $\{x,y\}\in E(G)$. For a vertex $x$ of $G$, we denote by $\mathcal{N}(x)$ the set of all the neighbours of $x$, also called the \textit{neighbourhood} of $x$. More precisely, $\mathcal{N}(x)=\{y\in V(G)\,:\,\{x,y\}\in E(G)\}$. Moreover, let $\mathcal{N}[x]=\mathcal{N}(x)\cup\{x\}$ be \textit{the closed neighbourhood of $x$}. \textit{The degree of the vertex $x$}, denoted by $\deg(x)$, is defined to be $\deg(x)=|\mathcal{N}(x)|$. By \textit{a free vertex} we mean a vertex of degree $1$. A vertex which is adjacent to a free vertex will be called a \textit{next-point}. A graph is called \textit{complete} if it has the property that any two vertices are adjacent. The complete graph with $n$ vertices is usually denoted by $\mathcal{K}_n$. By \textit{a subgraph} $H$ of $G$ we mean a graph with the property that $V(H)\subseteq V(G)$ and $E(H)\subseteq E(G)$. One says that a subgraph $H$ of $G$ is \textit{induced} if whenever $x,y\in V(H)$ so that $\{x,y\}\in E(G)$ then $\{x,y\}\in E(H)$. If $W\subseteq V$, we will denote by $G_W$ \textit{the induced subgraph of $G$ with the set of vertices $W$}. \textit{A clique} in $G$ is an induced subgraph which is a complete graph. For a vertex $x$, we denote by $C(x)$ the maximal clique containing $x$. A vertex $x$ is called \textit{cliqual} if there is at least one clique $C$ in $G$ so that $x\in V(C)$. Taking into account the number of cliques which contain a vertex, one says that a vertex is \textit{unicliqual} if there is in exactly one clique and it is \textit{multicliqual} if it is in more than one clique. Two vertices are \textit{cocliqual} if there is a clique containing both of them. A vertex is called \textit{neighbourly} if it is cocliqual with every vertex from its neighbourhood. \textit{A cut-point} is a vertex with the property that, after its removal, the number of connected components increases.  

\textit{A path of length }$t\geq2$ in $G$ is, by definition, a set of distinct vertices $x_0,x_1,\ldots,x_t$ such that $\{x_i,x_{i+1}\}$ are edges in $G$ for all $i\in\{0,\ldots,t-1\}$. \textit{The distance between two vertices $x$ and $y$}, denoted by $\dist_G(x,y)$, is defined to be the length of a shortest path joining $x$ and $y$. If there is no path joining $x$ and $y$, then $\dist_G(x,y)=\infty$. We will drop the subscript when the confusion is unlikely. One can define similarly the distance between two edges. If $e $ and $e'$ are two distinct edges, then \textit{the distance between the edges $e$ and $e'$} is denoted by $\dist(e,e')$ and is defined as the minimum integer $\ell$ such that there is a set of distinct edges $e=e_0,\ldots,e_{\ell}=e'$ such that $e_i\cap e_j\neq\emptyset$, for all $i\neq j$, $1\leq i,j \leq \ell$. \textit{The diameter of the graph $G$}, denoted by $\diam(G)$, is the maximum of all the distances between any two vertices in $G$, namely $$\diam(G)=\max\{\dist(x,y): \ x,y\in V(G)\}.$$

\textit{A cycle of length $n\geq3$}, usually denoted by $C_n$, is a graph with the vertex set $[n]=\{1,\ldots,n\}$ and the set of edges $\{i,i+1\}$, where $n+1=1$ by convention. A graph is \textit{chordal} if it doesn't have any induced cycles of length strictly greater than $3$. A graph is called a \textit{tree} if it is connected and it doesn't have cycles. Note that any tree is a chordal graph. Moreover the vertices of a tree are either free or cut-points.

For a graph $G$, let $G^c$ be \textit{the complement of the graph} $G$ that is the graph with the same vertex set as $G$ and $\{x,y\}$ is an edge of $G^c$ if it is not an edge of $G$. A graph $G$ is called \textit{co-chordal} if $G^c$ is a chordal graph. One says that the edges $\{x,y\}$ and $\{u,v\}$ of $G$ form \textit{an induced gap} in $G$ if $x,y,u,v$ are the vertices of an induced cycle of length $4$ in $G^c$. A graph $G$ is called \textit{gap-free} if it doesn't contain any induced gaps.

\textit{The square of the graph} $G$, denoted by $G^2$, is defined to be the graph with the same set of vertices as $G$ and with the edges \[E(G^2)=E(G)\cup\{\{x,y\}:\dist_G(x,y)=2\}.\] Harary and Ross \cite{RH} proved that the squares of trees are chordal graphs, a property which will be essential through the rest of the paper.
\subsection{Edge ideals}

Given a finite simple graph $G$ with the vertex set $V(G)=\{1,\ldots,n\}=[n]$ and the set of edges $E(G)$, one may consider its \textit{edge ideal} which is the squarefree monomial ideal denoted by $I(G)\subseteq S=\kk[x_1,\ldots,x_n]$, where $\kk$ is a field, defined by $I(G)=\langle x_ix_j\ :\ \{i,j\}\in E(G)\rangle $. Whenever it will be clear from the context, we will use the same notation for both variables and vertices. For a squarefree monomial $\mb=x_{i_1}\cdots x_{i_d}$, \textit{the support of the monomial $\mb$}, denoted by $\supp(\mb)$, is defined to be the set of all the variables dividing $\mb$, that is $\supp(\mb)=\{x_j:\ x_j\mid \mb\}=\{x_{i_1},\ldots,x_{i_d}\}$. Conversely, if $F=\{i_1,\ldots,i_d\}\subseteq [n]$, then the corresponding squarefree monomial is $\xb_F=\prod\limits_{j\in F}x_j=x_{i_1}\cdots x_{i_d}$, therefore $\supp\left(\xb_F\right)=F$. 

Edge ideals have been intensively studied and properties of invariants such that Betti numbers, projective dimension, Castelnuovo--Mumford regularity, or Krull dimension have been established for several classes of graphs (see \cite{HeHi,MV,Vi} for more details). We recall that, if $I\subseteq S=\kk[x_1,\ldots,x_n]$ is an ideal and $\mathcal{F}$ is the minimal graded free resolution of $S/I$ as an $S$-module:
\[\mathcal{F}: 0\rightarrow\bigoplus\limits_jS(-j)^{\beta_{pj}}\rightarrow\cdots\rightarrow\bigoplus\limits_j S(-j)^{\beta_{1j}}\rightarrow S\rightarrow S/I\rightarrow0,\] then \textit{the projective dimension of $S/I$} is \[\projdim \,S/I=\max\{i:\beta_{ij}\neq 0\}\] and \textit{the Castelnuovo--Mumford regularity} is \[\reg\, S/I=\max\{j-i:\beta_{ij}\neq0\}.\]
In the sequel we recall connections between Krull dimension, regularity and projective dimension of the edge ideal and combinatorial invariants of the graph. We start with the Krull dimension. 
A subset $W$ of $V(G)$ is called \textit{an independent set of $G$} if, for all $u,v\in W$ such that $u\neq v$, one has $\{u,v\}\notin E(G)$. 
It is well-known that one may compute the Krull dimension of $S/I(G)$ by using independent sets:
\begin{Proposition}\label{dimindep}
	$\dim S/I(G)=\max\{|W|:\ W\mbox{ is an independent set of }G\}.$
\end{Proposition}
A subset $M=\{e_1,\ldots,e_s\}$ of $E(G)$ is called \textit{a matching of $G$} if for all $i\neq j$, one has $e_i\cap e_j=\emptyset$. \textit{An induced matching} in $G$ is a matching which forms an induced subgraph of $G$.\textit{ The induced matching number} of $G$, denoted by $\indmat(G)$, is defined to be the number of edges in a largest induced matching, that is 
$$\indmat(G)=\max\{|M|:\ M\mbox{ is an induced matching in }G\}.$$
\textit{The co-chordal cover number of $G$}, denoted $\cochord(G)$, is the minimum number of co-chordal subgraphs required to cover the edges of $G$.

In between the induced matching number, the co-chordal cover number, and the Castelnuovo--Mumford regularity there are the following connections:
\begin{Proposition}\label{indmatgen}
	Let $G$ be a finite simple graph. Then
	\begin{itemize}
		\item[a)] $\reg\, S/I(G)\geq\indmat(G)$, \cite[Lemma 2.2]{Ka}.
		\item[b)] If $G$ is a chordal graph, then $\reg\, S/I(G)=\indmat(G)$, \cite[Corollary 6.9]{HaTu}.
		\item[c)] Over any field $\kk$, $\reg\, S/I(G)\leq\cochord(G)$ \cite[Theorem 1]{W} .
	\end{itemize}
\end{Proposition}
The following result describes the behaviour of the Castenuovo--Mumford re\-gularity with respect to induced subgraphs.
\begin{Proposition}\cite[Proposition 3.8]{MV}\label{regind}
	If $H$ is an induced subgraph of $G$, then $\reg\, S/I(H)\leq\reg\, S/I(G)$.
\end{Proposition}
\section{Properties of squares of graphs}
\label{sec:2}
In this section we recall several properties of squares of trees which will be used in the next section and we will prove a result which allows us to extend the results obtained in this paper to larger classes of graphs which are no longer trees.

Harary and Ross gave a complete characterization of the square of a tree, \cite{RH}. We recall here their characterization and we will emphasize those proper\-ties that will be used in the rest of the paper. 

We start with some properties of squares of trees.
\begin{Proposition}\cite{RH}\label{propT2}
	Let $T$ be a tree and $T^2$ its square. The following statements hold:
	\begin{itemize}
		\item[a)] The graph $T^2$ is a complete graph if and only if $T$ has exactly one next-point that is $T$ is a star graph. 
		\item[b)] Every clique of $T^2$ is the neighbourhood of a cut-point of $T$, and conversely.
		\item[c)] Assume that $T^2$ is not a complete graph. A vertex is a multicliqual point in $T^2$ if and only if it is a cut-point in $T$.
		\item[d)] Two cut-points of $T$ are adjacent in $T$ if and only if their neighbourhoods are cliques in $T^2$ that intersect in exactly two vertices.
		\item[e)] The graph $T^2$ has no cut-points.
	\end{itemize}
\end{Proposition}

The following theorem provides a characterization of graphs which are squares of trees. 
\begin{Theorem}\cite{RH}\label{charT2}
	Let $T$ be a tree which is not a star graph. Then $G=T^2$ if and only if the following conditions are fulfilled:
	\begin{itemize}
		\item[a)] Every vertex of $G$ is neighbourly and $G$ is connected.
		\item[b)] If two cliques meet at only one vertex $x$, then there is a third clique with which they share $x$ and exactly one other vertex.
		\item[c)] There is a one-to-one correspondence between the cliques and the multicliqual vertices $x$ of $G$ such that the clique $C(x)$ corresponding to $x$  contains exactly as many multicliqual vertices as the number of cliques which include $x$.
		\item[d)] No two cliques intersect in more than two vertices.
		\item[e)] The number of pairs of cliques that meet in two vertices is one less than the number of cliques. 
	\end{itemize}
\end{Theorem}

The following result evidences that all the results obtained in this paper are valid for larger classes of graphs, even if we will discuss mainly about squares of trees.
\begin{Proposition}\label{general}
	Let $G$ be a graph with the vertex set $V(G)$ and $E(G)$ its set of edges. Let $x\in V(G)$ and $y_1,y_2\in \mathcal{N}(x)$ be free vertices. Let $G_1$ be the graph obtained from $G$ by adding the edge $\{y_1,y_2\}$. Then $G_1^2=G^2$.
\end{Proposition}
\begin{proof} One may first note that the inclusion $E(G^2)\subseteq E(G_1^2)$ always holds. Moreover, $\{y_1,y_2\}\in E(G^2)$ since $\dist_G(y_1,y_2)=2$. Since in $G_1$ we added an edge between two free vertices, we have that $\dist_{G_1}(y_1,z)=\dist_{G}(y_1,z)$ and $\dist_{G_1}(y_2,z)=\dist_{G}(y_2,z)$ for every $z\in V(G)\setminus\{y_1,y_2\}$. Therefore $E(G_1^2)\setminus E(G^2)=\emptyset$.
\end{proof}

\section{Edge ideals of square of trees}
\label{sec:3}
Through this section we characterize all trees with the property that the edge ideal of their square has a linear resolution.  
In \cite{HT}, the authors defined the notion of ideals with linear quotients and proved that they have a linear resolution if they are generated in one degree. In general, there are ideals with a linear resolution which don't have linear quotients. However, for monomial ideals generated in degree $2$, these two notions are equivalent \cite[Theorem 3.2]{HHZ}. Therefore, in order to prove that the corresponding edge ideals have a linear resolution, we will prove that they have linear quotients.

We recall that a monomial ideal $I\subseteq S=\kk[x_1,\ldots,x_n]$, with the minimal system of generators $\mathcal{G}(I)=\{u_1,\ldots,u_r\}$ has \textit{linear quotients} if there is a monomial order of the generators $u_1\prec u_2\prec\cdots\prec u_r$ such that for all $2\leq j\leq r$, the colon ideal $(u_1,\ldots,u_{j-1}):(u_j)$ is generated by a set of variables.  We consider defined on $S$ the reverse lexicographical order, revlex for short, with respect to the order of the variables $x_1>x_2>\cdots>x_n$. We recall that, given two monomials $u=x_1^{a_1}\cdots x_n^{a_n}$ and $v=x_1^{b_1}\cdots x_n^{b_n}$, one says that $u>_{revlex}v$ if $\deg(u)>\deg(v)$ or $\deg(u)=\deg(v)$ and there is an integer $1\leq s\leq n$ such that $a_n=b_n,\ldots,a_{s+1}=b_{s+1}$ and $a_s<b_s$.
\begin{Remark}\rm
	According to the first statement of Proposition~\ref{propT2}, if $T$ is a star graph, then $T^2$ is a complete graph, therefore its edge ideal has a linear resolution.
\end{Remark}

In the sequel, we consider the class of (partially) whiskered stars, that are graphs obtained from the star graph by adding an edge (also called a whisker) to some of the free vertices. If we add a whisker to every free vertex we obtain a whiskered star. Let $T$ be the (partially) whiskered star with the vertex set $$V(T)=\{x_0,x_1,\ldots,x_{n},y_1,\ldots,y_{m}\},$$ where $1\leq m\leq n$. We assume that the set of edges of $T$ is $$E(T)=\{\{x_0,x_i\}:1\leq i\leq n\}\cup\{\{x_i,y_i\}:1\leq i\leq m\}.$$ Therefore $$E\left(T^2\right)=E(C(x_0))\cup\{\{x_0,y_i\}:1\leq i\leq m\}\cup\{\{x_i,y_i\}:1\leq i\leq m\}.$$ Let $S=\kk[x_0,x_1,\ldots,x_{n},y_1,\ldots,y_{m}]$ endowed with the revlex order with res\-pect to $x_0>x_1>\cdots>x_n>y_1>\cdots>y_m$ and $I(T^2)\subseteq S$. Note that we used the same notation for both vertices and variables. Assume that the mini\-mal set of monomial generators of $I(T^2)$ is $\mathcal{G}(I(T^2))=\{u_1,\ldots, u_N\},$ and $u_1>_{revlex}\cdots>_{revlex}u_N$.
\begin{Proposition}\label{wstar}
	Under the above assumptions, the following is true:
	$$\langle u_1,\ldots,u_{i-1}\rangle:\langle u_i\rangle=\left\{\begin{array}{cc}\langle x_0,\ldots,\hat{x}_a,\ldots,x_{b-1}\rangle,& \mbox{if } u_i=x_ax_b,\ 0\leq a<b\leq n\\
	\langle x_1,\ldots,x_n,y_1,\ldots,y_{b-1}\rangle,& \mbox{if } u_i=x_0y_b,\ 1\leq b\leq m\\
	\langle x_0,\ldots,\hat{x}_a,\ldots,x_n\rangle,& \mbox{if } u_i=x_ay_a,\ 1\leq a\leq m
	\end{array}\right.$$ for all $2\leq i\leq N$.
\end{Proposition}
\begin{proof} Let $u_i\in\mathcal{G}(I(T^2))$ for some $2\leq i\leq N$.
	We split the proof in three cases induced by the form of the monomial $u_i$.
	
	\textbf{Case 1:} Assume that $u_i=x_ax_b$. In order to prove that $$\langle x_0,\ldots,\hat{x}_a,\ldots,x_{b-1}\rangle\subseteq \langle u_1,\ldots,u_{i-1}\rangle:\langle u_i\rangle,$$ let $0\leq j\leq b-1$, $j\neq a$. Since $x_jx_a>_{revlex}u_i$ and $x_jx_a$ is a squarefree monomial, one has $x_ju_{i}\in\langle u_1,\ldots,u_{i-1} \rangle$, that is $x_{j}\in\langle u_1,\ldots,u_{i-1}\rangle:\langle u_i\rangle$. For the reverse inclusion, let $w\in\langle u_1,\ldots,u_{i-1}\rangle:\langle u_i\rangle$ and we have to prove that $w\in\langle x_0,\ldots,\hat{x}_a,\ldots,x_{b-1}\rangle$. Since we are dealing with monomial ideals, it is enough to consider that $w$ is a monomial. Note that $u_i=x_ax_b$ implies that $\supp(u_j)\subseteq\{x_0,\ldots,x_n\}$, for all $1\leq j\leq i-1$. Then $w u_i\in\langle u_1,\ldots,u_{i-1}\rangle$ implies that there is $j\in\{0,\ldots,i-1\}$ so that $u_j\mid w u_i$. Hence there is $x_\alpha\mid u_j$ and $x_\alpha\nmid u_i$ so that $x_{\alpha}\mid w$. Note that $\alpha\leq\max(u_j)< b$ since $u_j>_{revlex}u_i$. Therefore $\alpha\in\{0,\ldots,b-1\}$. The statement follows.
	
	\textbf{Case 2:} We assume now that $u_i=x_0y_b$, $1\leq b\leq m$ and we use the same strategy as before. Our first goal is to prove that $$\langle x_1,\ldots,x_n,y_1,\ldots,y_{b-1}\rangle\subseteq \langle u_1,\ldots,u_{i-1}\rangle:\langle u_i\rangle.$$ Note that if $1\leq j\leq b-1$ then the monomial $x_0y_j=y_ju_i/y_b>_{revlex}u_i$ and is squarefree, hence $x_ju_{i}\in\langle u_1,\ldots,u_{i-1} \rangle$, that is $y_{j}\in\langle u_1,\ldots,u_{i-1}\rangle:\langle u_i\rangle$. Also, if $1\leq j\leq n$ then the monomial $x_0x_j=x_ju_i/y_b>_{revlex}u_i$ and is squarefree, hence $x_ju_{i}\in\langle u_1,\ldots,u_{i-1} \rangle$, that is $x_{j}\in\langle u_1,\ldots,u_{i-1}\rangle:\langle u_i\rangle$. The second goal is to prove the converse inclusion. Let $w\in\langle u_1,\ldots,u_{i-1}\rangle:\langle u_i\rangle$. We have to prove that $w\in\langle x_1,\ldots,x_n,y_1,\ldots,y_{b-1}\rangle$. As before, it is enough to consider that $w$ is a monomial. Since $w u_i\in\langle u_1,\ldots,u_{i-1}\rangle$, there is a monomial $u_j\mid w u_i$. Since $j\leq i-1$, $u_j$ has one of the following cases: $u_j=x_{\alpha}x_{\beta}$ or $u_j=x_0y_{l}$, with $1\leq l\leq b-1$. In the first case, $\beta>\alpha\geq0$ and $x_{\beta}\mid u_{j}\mid w u_i$, but $x_{\beta}\nmid u_i=x_0y_b$, therefore $x_{\beta}\mid w$. In the second case $y_l\mid u_j\mid w u_i$ and $y_l\nmid u_i=x_0y_b$, hence $y_l\mid w$. In both cases one has that $w\in\langle x_1,\ldots,x_n,y_1,\ldots,y_{b-1}\rangle$. 
	
	\textbf{Case 3:} Assume now that $u_i=x_ay_a$ with $1\leq a\leq m$. Note that for $0\leq j\leq n$, $j\neq a$ one has $x_ju_i=x_jx_ay_a$, $x_jx_a>_{revlex}u_i$ and is a squarefree monomial ideal. Therefore $x_ju_i\in\langle u_1,\ldots,u_{i-1}\rangle$. For the other inclusion, let $w\in\langle u_1,\ldots,u_{i-1}\rangle:\langle u_i\rangle$ and we have to prove that $w\in\langle x_0,\ldots,\hat{x}_a,\ldots,x_n\rangle$. Since $w\in\langle u_1,\ldots,u_{i-1}\rangle:\langle u_i\rangle$, there is some $j$, $1\leq j\leq i-1$ such that $u_j\mid w u_i$. Since $u_j>_{revlex}u_i$ and $u_i=x_ay_a$, there is some $\ell\neq a$ such that $x_\ell\mid u_j$ and $x_\ell\nmid u_i$. As $u_j\mid w u_i$, one must have $x_\ell\mid w$. Therefore, $w\in\langle x_0,\ldots,\hat{x}_a,\ldots,x_n\rangle$.
\end{proof}
The next result follows directly from Proposition~\ref{wstar} and \cite[Theorem 3.2]{HHZ}.

\begin{Theorem}
	If $T$ is a (partially) whiskered star, then $I(T^2)$ has a linear resolution.
\end{Theorem}
For a monomial ideal with linear quotients $I$ with respect to the order of the generators $u_1\succ\cdots\succ u_N$, one denotes $$\set(u_i)=\{x_j:x_j\in\langle u_1,\ldots,u_{i-1}\rangle:\langle u_i\rangle\},$$and $r_i=|\set(u_i)|$ for all $i$, $1\leq i\leq N$.

One may compute the Betti numbers of a monomial ideal with linear quotients by using the above notions, as the following result states. 
\begin{Proposition}\cite{HT}
	Let $I\subseteq S$ be a monomial ideal with linear quotients. Then $$\beta_i(I)=\sum_{k=1}^N {r_k \choose i}.$$In particular, it follows that $$\projdim(I)=\max\{r_1,\ldots,r_N\}.$$
\end{Proposition}
The following result is now straightforward:
\begin{Corollary}
	Let $T$ be a (partially) whiskered star graph and $I\left(T^2\right)$ the edge ideal of its square. Then $\projdim(I\left(T^2\right))=M-2$, where $M$ is the number of vertices of $T$.
\end{Corollary}
\begin{proof} Taking into account Proposition~\ref{wstar}, the maximum of the set $\{r_1,\ldots,r_N\}$ is achieved for the monomial $u_N=x_my_m$ for which $r_N=n+m-1$. Note that $T$ is a graph on $M=n+m+1$ vertices.  
\end{proof}

In order to characterize all the trees such that the edge ideal of their square has a linear resolution, we will use the combinatorial characterization of an edge ideal with a linear resolution given by Fr\"oberg in \cite{F}.
\begin{Theorem}\cite{F} \label{Froberg} Let $G$ be a finite simple graph. The edge ideal $I(G)$ has a linear resolution if and only if $G$ is a co-chordal graph.
\end{Theorem}
The following remark is a direct consequence of Theorem~\ref{Froberg} and it will be intensively used through the paper:
\begin{Remark}\label{gapfree}\rm
	If $G$ is a graph such that its edge ideal has a linear resolution, then $G$ is gap-free.
\end{Remark}

Even if one would expect that the property of having a linear resolution is preserved by the squares, the following example shows that this is not true:

\begin{Example}\label{example}\rm 
	Let $T$ be the tree from the figure below and $T^2$ its square.
	\newline
	\begin{figure}[h]% Graphic for TeX using PGF
		% Title: /home/anda/Diagram1.dia
		% Creator: Dia v0.97+git
		% CreationDate: Wed Sep 25 13:15:15 2019
		% For: anda
		% \usepackage{tikz}
		% The following commands are not supported in PSTricks at present
		% We define them conditionally, so when they are implemented,
		% this pgf file will use them.
		\ifx\du\undefined
		\newlength{\du}
		\fi
		\setlength{\du}{7\unitlength}
		\begin{tikzpicture}[even odd rule]
		\pgftransformxscale{1.000000}
		\pgftransformyscale{-1.000000}
		\definecolor{dialinecolor}{rgb}{0.000000, 0.000000, 0.000000}
		\pgfsetstrokecolor{dialinecolor}
		\pgfsetstrokeopacity{1.000000}
		\definecolor{diafillcolor}{rgb}{1.000000, 1.000000, 1.000000}
		\pgfsetfillcolor{diafillcolor}
		\pgfsetfillopacity{1.000000}
		\pgfsetlinewidth{0.0500000\du}
		\pgfsetdash{}{0pt}
		\pgfsetbuttcap
		{
			\definecolor{diafillcolor}{rgb}{0.000000, 0.000000, 0.000000}
			\pgfsetfillcolor{diafillcolor}
			\pgfsetfillopacity{1.000000}
			% was here!!!
			\definecolor{dialinecolor}{rgb}{0.000000, 0.000000, 0.000000}
			\pgfsetstrokecolor{dialinecolor}
			\pgfsetstrokeopacity{1.000000}
			\draw (2.000000\du,2.000000\du)--(7.000000\du,5.000000\du);
		}
		\definecolor{dialinecolor}{rgb}{0.000000, 0.000000, 0.000000}
		\pgfsetstrokecolor{dialinecolor}
		\pgfsetstrokeopacity{1.000000}
		\draw (2.000000\du,2.000000\du)--(7.000000\du,5.000000\du);
		\pgfsetlinewidth{0.0500000\du}
		\pgfsetdash{}{0pt}
		\pgfsetmiterjoin
		\pgfsetbuttcap
		\definecolor{diafillcolor}{rgb}{0.000000, 0.000000, 0.000000}
		\pgfsetfillcolor{diafillcolor}
		\pgfsetfillopacity{1.000000}
		\definecolor{dialinecolor}{rgb}{0.000000, 0.000000, 0.000000}
		\pgfsetstrokecolor{dialinecolor}
		\pgfsetstrokeopacity{1.000000}
		\pgfpathmoveto{\pgfpoint{2.000000\du}{2.000000\du}}
		\pgfpathcurveto{\pgfpoint{2.064312\du}{1.892813\du}}{\pgfpoint{2.235811\du}{1.849939\du}}{\pgfpoint{2.342997\du}{1.914251\du}}
		\pgfpathcurveto{\pgfpoint{2.450184\du}{1.978563\du}}{\pgfpoint{2.493058\du}{2.150061\du}}{\pgfpoint{2.428746\du}{2.257248\du}}
		\pgfpathcurveto{\pgfpoint{2.364434\du}{2.364434\du}}{\pgfpoint{2.192936\du}{2.407309\du}}{\pgfpoint{2.085749\du}{2.342997\du}}
		\pgfpathcurveto{\pgfpoint{1.978563\du}{2.278685\du}}{\pgfpoint{1.935688\du}{2.107187\du}}{\pgfpoint{2.000000\du}{2.000000\du}}
		\pgfpathclose
		\pgfusepath{fill,stroke}
		\pgfsetlinewidth{0.0500000\du}
		\pgfsetdash{}{0pt}
		\pgfsetbuttcap
		{
			\definecolor{diafillcolor}{rgb}{0.000000, 0.000000, 0.000000}
			\pgfsetfillcolor{diafillcolor}
			\pgfsetfillopacity{1.000000}
			% was here!!!
			\definecolor{dialinecolor}{rgb}{0.000000, 0.000000, 0.000000}
			\pgfsetstrokecolor{dialinecolor}
			\pgfsetstrokeopacity{1.000000}
			\draw (2.000000\du,8.000000\du)--(7.000000\du,5.000000\du);
		}
		\definecolor{dialinecolor}{rgb}{0.000000, 0.000000, 0.000000}
		\pgfsetstrokecolor{dialinecolor}
		\pgfsetstrokeopacity{1.000000}
		\draw (2.000000\du,8.000000\du)--(7.000000\du,5.000000\du);
		\pgfsetlinewidth{0.0500000\du}
		\pgfsetdash{}{0pt}
		\pgfsetmiterjoin
		\pgfsetbuttcap
		\definecolor{diafillcolor}{rgb}{0.000000, 0.000000, 0.000000}
		\pgfsetfillcolor{diafillcolor}
		\pgfsetfillopacity{1.000000}
		\definecolor{dialinecolor}{rgb}{0.000000, 0.000000, 0.000000}
		\pgfsetstrokecolor{dialinecolor}
		\pgfsetstrokeopacity{1.000000}
		\pgfpathmoveto{\pgfpoint{2.000000\du}{8.000000\du}}
		\pgfpathcurveto{\pgfpoint{1.935688\du}{7.892813\du}}{\pgfpoint{1.978563\du}{7.721315\du}}{\pgfpoint{2.085749\du}{7.657003\du}}
		\pgfpathcurveto{\pgfpoint{2.192936\du}{7.592691\du}}{\pgfpoint{2.364434\du}{7.635566\du}}{\pgfpoint{2.428746\du}{7.742752\du}}
		\pgfpathcurveto{\pgfpoint{2.493058\du}{7.849939\du}}{\pgfpoint{2.450184\du}{8.021437\du}}{\pgfpoint{2.342997\du}{8.085749\du}}
		\pgfpathcurveto{\pgfpoint{2.235811\du}{8.150061\du}}{\pgfpoint{2.064312\du}{8.107187\du}}{\pgfpoint{2.000000\du}{8.000000\du}}
		\pgfpathclose
		\pgfusepath{fill,stroke}
		\pgfsetlinewidth{0.0500000\du}
		\pgfsetdash{}{0pt}
		\pgfsetbuttcap
		{
			\definecolor{diafillcolor}{rgb}{0.000000, 0.000000, 0.000000}
			\pgfsetfillcolor{diafillcolor}
			\pgfsetfillopacity{1.000000}
			% was here!!!
			\definecolor{dialinecolor}{rgb}{0.000000, 0.000000, 0.000000}
			\pgfsetstrokecolor{dialinecolor}
			\pgfsetstrokeopacity{1.000000}
			\draw (7.000000\du,5.000000\du)--(13.000000\du,5.000000\du);
		}
		\definecolor{dialinecolor}{rgb}{0.000000, 0.000000, 0.000000}
		\pgfsetstrokecolor{dialinecolor}
		\pgfsetstrokeopacity{1.000000}
		\draw (7.000000\du,5.000000\du)--(13.000000\du,5.000000\du);
		\pgfsetlinewidth{0.0500000\du}
		\pgfsetdash{}{0pt}
		\pgfsetmiterjoin
		\pgfsetbuttcap
		\definecolor{diafillcolor}{rgb}{0.000000, 0.000000, 0.000000}
		\pgfsetfillcolor{diafillcolor}
		\pgfsetfillopacity{1.000000}
		\definecolor{dialinecolor}{rgb}{0.000000, 0.000000, 0.000000}
		\pgfsetstrokecolor{dialinecolor}
		\pgfsetstrokeopacity{1.000000}
		\pgfpathmoveto{\pgfpoint{7.000000\du}{5.000000\du}}
		\pgfpathcurveto{\pgfpoint{7.000000\du}{4.875000\du}}{\pgfpoint{7.125000\du}{4.750000\du}}{\pgfpoint{7.250000\du}{4.750000\du}}
		\pgfpathcurveto{\pgfpoint{7.375000\du}{4.750000\du}}{\pgfpoint{7.500000\du}{4.875000\du}}{\pgfpoint{7.500000\du}{5.000000\du}}
		\pgfpathcurveto{\pgfpoint{7.500000\du}{5.125000\du}}{\pgfpoint{7.375000\du}{5.250000\du}}{\pgfpoint{7.250000\du}{5.250000\du}}
		\pgfpathcurveto{\pgfpoint{7.125000\du}{5.250000\du}}{\pgfpoint{7.000000\du}{5.125000\du}}{\pgfpoint{7.000000\du}{5.000000\du}}
		\pgfpathclose
		\pgfusepath{fill,stroke}
		\pgfsetlinewidth{0.0500000\du}
		\pgfsetdash{}{0pt}
		\pgfsetbuttcap
		{
			\definecolor{diafillcolor}{rgb}{0.000000, 0.000000, 0.000000}
			\pgfsetfillcolor{diafillcolor}
			\pgfsetfillopacity{1.000000}
			% was here!!!
		}
		\definecolor{dialinecolor}{rgb}{0.000000, 0.000000, 0.000000}
		\pgfsetstrokecolor{dialinecolor}
		\pgfsetstrokeopacity{1.000000}
		\draw (13.000000\du,5.000000\du)--(18.000000\du,2.000000\du);
		\pgfsetlinewidth{0.0500000\du}
		\pgfsetdash{}{0pt}
		\pgfsetmiterjoin
		\pgfsetbuttcap
		\definecolor{diafillcolor}{rgb}{0.000000, 0.000000, 0.000000}
		\pgfsetfillcolor{diafillcolor}
		\pgfsetfillopacity{1.000000}
		\definecolor{dialinecolor}{rgb}{0.000000, 0.000000, 0.000000}
		\pgfsetstrokecolor{dialinecolor}
		\pgfsetstrokeopacity{1.000000}
		\pgfpathmoveto{\pgfpoint{13.000000\du}{5.000000\du}}
		\pgfpathcurveto{\pgfpoint{12.935688\du}{4.892813\du}}{\pgfpoint{12.978563\du}{4.721315\du}}{\pgfpoint{13.085749\du}{4.657003\du}}
		\pgfpathcurveto{\pgfpoint{13.192936\du}{4.592691\du}}{\pgfpoint{13.364434\du}{4.635566\du}}{\pgfpoint{13.428746\du}{4.742752\du}}
		\pgfpathcurveto{\pgfpoint{13.493058\du}{4.849939\du}}{\pgfpoint{13.450184\du}{5.021437\du}}{\pgfpoint{13.342997\du}{5.085749\du}}
		\pgfpathcurveto{\pgfpoint{13.235811\du}{5.150061\du}}{\pgfpoint{13.064312\du}{5.107187\du}}{\pgfpoint{13.000000\du}{5.000000\du}}
		\pgfpathclose
		\pgfusepath{fill,stroke}
		\pgfsetlinewidth{0.0500000\du}
		\pgfsetdash{}{0pt}
		\pgfsetmiterjoin
		\pgfsetbuttcap
		\definecolor{diafillcolor}{rgb}{0.000000, 0.000000, 0.000000}
		\pgfsetfillcolor{diafillcolor}
		\pgfsetfillopacity{1.000000}
		\definecolor{dialinecolor}{rgb}{0.000000, 0.000000, 0.000000}
		\pgfsetstrokecolor{dialinecolor}
		\pgfsetstrokeopacity{1.000000}
		\pgfpathmoveto{\pgfpoint{18.000000\du}{2.000000\du}}
		\pgfpathcurveto{\pgfpoint{18.064312\du}{2.107187\du}}{\pgfpoint{18.021437\du}{2.278685\du}}{\pgfpoint{17.914251\du}{2.342997\du}}
		\pgfpathcurveto{\pgfpoint{17.807064\du}{2.407309\du}}{\pgfpoint{17.635566\du}{2.364434\du}}{\pgfpoint{17.571254\du}{2.257248\du}}
		\pgfpathcurveto{\pgfpoint{17.506942\du}{2.150061\du}}{\pgfpoint{17.549816\du}{1.978563\du}}{\pgfpoint{17.657003\du}{1.914251\du}}
		\pgfpathcurveto{\pgfpoint{17.764189\du}{1.849939\du}}{\pgfpoint{17.935688\du}{1.892813\du}}{\pgfpoint{18.000000\du}{2.000000\du}}
		\pgfpathclose
		\pgfusepath{fill,stroke}
		\pgfsetlinewidth{0.0500000\du}
		\pgfsetdash{}{0pt}
		\pgfsetbuttcap
		{
			\definecolor{diafillcolor}{rgb}{0.000000, 0.000000, 0.000000}
			\pgfsetfillcolor{diafillcolor}
			\pgfsetfillopacity{1.000000}
			% was here!!!
			\definecolor{dialinecolor}{rgb}{0.000000, 0.000000, 0.000000}
			\pgfsetstrokecolor{dialinecolor}
			\pgfsetstrokeopacity{1.000000}
			\draw (13.000000\du,5.000000\du)--(18.000000\du,8.000000\du);
		}
		\definecolor{dialinecolor}{rgb}{0.000000, 0.000000, 0.000000}
		\pgfsetstrokecolor{dialinecolor}
		\pgfsetstrokeopacity{1.000000}
		\draw (13.000000\du,5.000000\du)--(18.000000\du,8.000000\du);
		\pgfsetlinewidth{0.0500000\du}
		\pgfsetdash{}{0pt}
		\pgfsetmiterjoin
		\pgfsetbuttcap
		\definecolor{diafillcolor}{rgb}{0.000000, 0.000000, 0.000000}
		\pgfsetfillcolor{diafillcolor}
		\pgfsetfillopacity{1.000000}
		\definecolor{dialinecolor}{rgb}{0.000000, 0.000000, 0.000000}
		\pgfsetstrokecolor{dialinecolor}
		\pgfsetstrokeopacity{1.000000}
		\pgfpathmoveto{\pgfpoint{18.000000\du}{8.000000\du}}
		\pgfpathcurveto{\pgfpoint{17.935688\du}{8.107187\du}}{\pgfpoint{17.764189\du}{8.150061\du}}{\pgfpoint{17.657003\du}{8.085749\du}}
		\pgfpathcurveto{\pgfpoint{17.549816\du}{8.021437\du}}{\pgfpoint{17.506942\du}{7.849939\du}}{\pgfpoint{17.571254\du}{7.742752\du}}
		\pgfpathcurveto{\pgfpoint{17.635566\du}{7.635566\du}}{\pgfpoint{17.807064\du}{7.592691\du}}{\pgfpoint{17.914251\du}{7.657003\du}}
		\pgfpathcurveto{\pgfpoint{18.021437\du}{7.721315\du}}{\pgfpoint{18.064312\du}{7.892813\du}}{\pgfpoint{18.000000\du}{8.000000\du}}
		\pgfpathclose
		\pgfusepath{fill,stroke}
		\node[anchor=west] at (-3\du,5\du){$T:$};
		\node[anchor=west] at (0\du,8\du){$1$};
		\node[anchor=west] at (0\du,2\du){$2$};
		\node[anchor=west] at (18\du,8\du){$6$};
		\node[anchor=west] at (18\du,2\du){$5$};
		\node[anchor=west] at (4.7\du,5\du){$3$};
		\node[anchor=west] at (13.5\du,5\du){$4$};
		\end{tikzpicture}
		% Graphic for TeX using PGF
		% Title: /home/anda/Diagram1.dia
		% Creator: Dia v0.97+git
		% CreationDate: Wed Sep 25 13:16:14 2019
		% For: anda
		% \usepackage{tikz}
		% The following commands are not supported in PSTricks at present
		% We define them conditionally, so when they are implemented,
		% this pgf file will use them.
		\ifx\du\undefined
		\newlength{\du}
		\fi
		\setlength{\du}{7\unitlength}
		\begin{tikzpicture}[even odd rule]
		\pgftransformxscale{1.000000}
		\pgftransformyscale{-1.000000}
		\definecolor{dialinecolor}{rgb}{0.000000, 0.000000, 0.000000}
		\pgfsetstrokecolor{dialinecolor}
		\pgfsetstrokeopacity{1.000000}
		\definecolor{diafillcolor}{rgb}{1.000000, 1.000000, 1.000000}
		\pgfsetfillcolor{diafillcolor}
		\pgfsetfillopacity{1.000000}
		\pgfsetlinewidth{0.0500000\du}
		\pgfsetdash{}{0pt}
		\pgfsetbuttcap
		{
			\definecolor{diafillcolor}{rgb}{0.000000, 0.000000, 0.000000}
			\pgfsetfillcolor{diafillcolor}
			\pgfsetfillopacity{1.000000}
			% was here!!!
			\definecolor{dialinecolor}{rgb}{0.000000, 0.000000, 0.000000}
			\pgfsetstrokecolor{dialinecolor}
			\pgfsetstrokeopacity{1.000000}
			\draw (2.000000\du,2.000000\du)--(7.000000\du,5.000000\du);
		}
		\definecolor{dialinecolor}{rgb}{0.000000, 0.000000, 0.000000}
		\pgfsetstrokecolor{dialinecolor}
		\pgfsetstrokeopacity{1.000000}
		\draw (2.000000\du,2.000000\du)--(7.000000\du,5.000000\du);
		\pgfsetlinewidth{0.0500000\du}
		\pgfsetdash{}{0pt}
		\pgfsetmiterjoin
		\pgfsetbuttcap
		\definecolor{diafillcolor}{rgb}{0.000000, 0.000000, 0.000000}
		\pgfsetfillcolor{diafillcolor}
		\pgfsetfillopacity{1.000000}
		\definecolor{dialinecolor}{rgb}{0.000000, 0.000000, 0.000000}
		\pgfsetstrokecolor{dialinecolor}
		\pgfsetstrokeopacity{1.000000}
		\pgfpathmoveto{\pgfpoint{2.000000\du}{2.000000\du}}
		\pgfpathcurveto{\pgfpoint{2.064312\du}{1.892813\du}}{\pgfpoint{2.235811\du}{1.849939\du}}{\pgfpoint{2.342997\du}{1.914251\du}}
		\pgfpathcurveto{\pgfpoint{2.450184\du}{1.978563\du}}{\pgfpoint{2.493058\du}{2.150061\du}}{\pgfpoint{2.428746\du}{2.257248\du}}
		\pgfpathcurveto{\pgfpoint{2.364434\du}{2.364434\du}}{\pgfpoint{2.192936\du}{2.407309\du}}{\pgfpoint{2.085749\du}{2.342997\du}}
		\pgfpathcurveto{\pgfpoint{1.978563\du}{2.278685\du}}{\pgfpoint{1.935688\du}{2.107187\du}}{\pgfpoint{2.000000\du}{2.000000\du}}
		\pgfpathclose
		\pgfusepath{fill,stroke}
		\pgfsetlinewidth{0.0500000\du}
		\pgfsetdash{}{0pt}
		\pgfsetbuttcap
		{
			\definecolor{diafillcolor}{rgb}{0.000000, 0.000000, 0.000000}
			\pgfsetfillcolor{diafillcolor}
			\pgfsetfillopacity{1.000000}
			% was here!!!
			\definecolor{dialinecolor}{rgb}{0.000000, 0.000000, 0.000000}
			\pgfsetstrokecolor{dialinecolor}
			\pgfsetstrokeopacity{1.000000}
			\draw (2.000000\du,8.000000\du)--(7.000000\du,5.000000\du);
		}
		\definecolor{dialinecolor}{rgb}{0.000000, 0.000000, 0.000000}
		\pgfsetstrokecolor{dialinecolor}
		\pgfsetstrokeopacity{1.000000}
		\draw (2.000000\du,8.000000\du)--(7.000000\du,5.000000\du);
		\pgfsetlinewidth{0.0500000\du}
		\pgfsetdash{}{0pt}
		\pgfsetmiterjoin
		\pgfsetbuttcap
		\definecolor{diafillcolor}{rgb}{0.000000, 0.000000, 0.000000}
		\pgfsetfillcolor{diafillcolor}
		\pgfsetfillopacity{1.000000}
		\definecolor{dialinecolor}{rgb}{0.000000, 0.000000, 0.000000}
		\pgfsetstrokecolor{dialinecolor}
		\pgfsetstrokeopacity{1.000000}
		\pgfpathmoveto{\pgfpoint{2.000000\du}{8.000000\du}}
		\pgfpathcurveto{\pgfpoint{1.935688\du}{7.892813\du}}{\pgfpoint{1.978563\du}{7.721315\du}}{\pgfpoint{2.085749\du}{7.657003\du}}
		\pgfpathcurveto{\pgfpoint{2.192936\du}{7.592691\du}}{\pgfpoint{2.364434\du}{7.635566\du}}{\pgfpoint{2.428746\du}{7.742752\du}}
		\pgfpathcurveto{\pgfpoint{2.493058\du}{7.849939\du}}{\pgfpoint{2.450184\du}{8.021437\du}}{\pgfpoint{2.342997\du}{8.085749\du}}
		\pgfpathcurveto{\pgfpoint{2.235811\du}{8.150061\du}}{\pgfpoint{2.064312\du}{8.107187\du}}{\pgfpoint{2.000000\du}{8.000000\du}}
		\pgfpathclose
		\pgfusepath{fill,stroke}
		\pgfsetlinewidth{0.0500000\du}
		\pgfsetdash{}{0pt}
		\pgfsetbuttcap
		{
			\definecolor{diafillcolor}{rgb}{0.000000, 0.000000, 0.000000}
			\pgfsetfillcolor{diafillcolor}
			\pgfsetfillopacity{1.000000}
			% was here!!!
			\definecolor{dialinecolor}{rgb}{0.000000, 0.000000, 0.000000}
			\pgfsetstrokecolor{dialinecolor}
			\pgfsetstrokeopacity{1.000000}
			\draw (7.000000\du,5.000000\du)--(13.000000\du,5.000000\du);
		}
		\definecolor{dialinecolor}{rgb}{0.000000, 0.000000, 0.000000}
		\pgfsetstrokecolor{dialinecolor}
		\pgfsetstrokeopacity{1.000000}
		\draw (7.000000\du,5.000000\du)--(13.000000\du,5.000000\du);
		\pgfsetlinewidth{0.0500000\du}
		\pgfsetdash{}{0pt}
		\pgfsetmiterjoin
		\pgfsetbuttcap
		\definecolor{diafillcolor}{rgb}{0.000000, 0.000000, 0.000000}
		\pgfsetfillcolor{diafillcolor}
		\pgfsetfillopacity{1.000000}
		\definecolor{dialinecolor}{rgb}{0.000000, 0.000000, 0.000000}
		\pgfsetstrokecolor{dialinecolor}
		\pgfsetstrokeopacity{1.000000}
		\pgfpathmoveto{\pgfpoint{7.000000\du}{5.000000\du}}
		\pgfpathcurveto{\pgfpoint{7.000000\du}{4.875000\du}}{\pgfpoint{7.125000\du}{4.750000\du}}{\pgfpoint{7.250000\du}{4.750000\du}}
		\pgfpathcurveto{\pgfpoint{7.375000\du}{4.750000\du}}{\pgfpoint{7.500000\du}{4.875000\du}}{\pgfpoint{7.500000\du}{5.000000\du}}
		\pgfpathcurveto{\pgfpoint{7.500000\du}{5.125000\du}}{\pgfpoint{7.375000\du}{5.250000\du}}{\pgfpoint{7.250000\du}{5.250000\du}}
		\pgfpathcurveto{\pgfpoint{7.125000\du}{5.250000\du}}{\pgfpoint{7.000000\du}{5.125000\du}}{\pgfpoint{7.000000\du}{5.000000\du}}
		\pgfpathclose
		\pgfusepath{fill,stroke}
		\pgfsetlinewidth{0.0500000\du}
		\pgfsetdash{}{0pt}
		\pgfsetbuttcap
		{
			\definecolor{diafillcolor}{rgb}{0.000000, 0.000000, 0.000000}
			\pgfsetfillcolor{diafillcolor}
			\pgfsetfillopacity{1.000000}
			% was here!!!
		}
		\definecolor{dialinecolor}{rgb}{0.000000, 0.000000, 0.000000}
		\pgfsetstrokecolor{dialinecolor}
		\pgfsetstrokeopacity{1.000000}
		\draw (13.000000\du,5.000000\du)--(18.000000\du,2.000000\du);
		\pgfsetlinewidth{0.0500000\du}
		\pgfsetdash{}{0pt}
		\pgfsetmiterjoin
		\pgfsetbuttcap
		\definecolor{diafillcolor}{rgb}{0.000000, 0.000000, 0.000000}
		\pgfsetfillcolor{diafillcolor}
		\pgfsetfillopacity{1.000000}
		\definecolor{dialinecolor}{rgb}{0.000000, 0.000000, 0.000000}
		\pgfsetstrokecolor{dialinecolor}
		\pgfsetstrokeopacity{1.000000}
		\pgfpathmoveto{\pgfpoint{13.000000\du}{5.000000\du}}
		\pgfpathcurveto{\pgfpoint{12.935688\du}{4.892813\du}}{\pgfpoint{12.978563\du}{4.721315\du}}{\pgfpoint{13.085749\du}{4.657003\du}}
		\pgfpathcurveto{\pgfpoint{13.192936\du}{4.592691\du}}{\pgfpoint{13.364434\du}{4.635566\du}}{\pgfpoint{13.428746\du}{4.742752\du}}
		\pgfpathcurveto{\pgfpoint{13.493058\du}{4.849939\du}}{\pgfpoint{13.450184\du}{5.021437\du}}{\pgfpoint{13.342997\du}{5.085749\du}}
		\pgfpathcurveto{\pgfpoint{13.235811\du}{5.150061\du}}{\pgfpoint{13.064312\du}{5.107187\du}}{\pgfpoint{13.000000\du}{5.000000\du}}
		\pgfpathclose
		\pgfusepath{fill,stroke}
		\pgfsetlinewidth{0.0500000\du}
		\pgfsetdash{}{0pt}
		\pgfsetmiterjoin
		\pgfsetbuttcap
		\definecolor{diafillcolor}{rgb}{0.000000, 0.000000, 0.000000}
		\pgfsetfillcolor{diafillcolor}
		\pgfsetfillopacity{1.000000}
		\definecolor{dialinecolor}{rgb}{0.000000, 0.000000, 0.000000}
		\pgfsetstrokecolor{dialinecolor}
		\pgfsetstrokeopacity{1.000000}
		\pgfpathmoveto{\pgfpoint{18.000000\du}{2.000000\du}}
		\pgfpathcurveto{\pgfpoint{18.064312\du}{2.107187\du}}{\pgfpoint{18.021437\du}{2.278685\du}}{\pgfpoint{17.914251\du}{2.342997\du}}
		\pgfpathcurveto{\pgfpoint{17.807064\du}{2.407309\du}}{\pgfpoint{17.635566\du}{2.364434\du}}{\pgfpoint{17.571254\du}{2.257248\du}}
		\pgfpathcurveto{\pgfpoint{17.506942\du}{2.150061\du}}{\pgfpoint{17.549816\du}{1.978563\du}}{\pgfpoint{17.657003\du}{1.914251\du}}
		\pgfpathcurveto{\pgfpoint{17.764189\du}{1.849939\du}}{\pgfpoint{17.935688\du}{1.892813\du}}{\pgfpoint{18.000000\du}{2.000000\du}}
		\pgfpathclose
		\pgfusepath{fill,stroke}
		\pgfsetlinewidth{0.0500000\du}
		\pgfsetdash{}{0pt}
		\pgfsetbuttcap
		{
			\definecolor{diafillcolor}{rgb}{0.000000, 0.000000, 0.000000}
			\pgfsetfillcolor{diafillcolor}
			\pgfsetfillopacity{1.000000}
			% was here!!!
			\definecolor{dialinecolor}{rgb}{0.000000, 0.000000, 0.000000}
			\pgfsetstrokecolor{dialinecolor}
			\pgfsetstrokeopacity{1.000000}
			\draw (13.000000\du,5.000000\du)--(18.000000\du,8.000000\du);
		}
		\definecolor{dialinecolor}{rgb}{0.000000, 0.000000, 0.000000}
		\pgfsetstrokecolor{dialinecolor}
		\pgfsetstrokeopacity{1.000000}
		\draw (13.000000\du,5.000000\du)--(18.000000\du,8.000000\du);
		\pgfsetlinewidth{0.0500000\du}
		\pgfsetdash{}{0pt}
		\pgfsetmiterjoin
		\pgfsetbuttcap
		\definecolor{diafillcolor}{rgb}{0.000000, 0.000000, 0.000000}
		\pgfsetfillcolor{diafillcolor}
		\pgfsetfillopacity{1.000000}
		\definecolor{dialinecolor}{rgb}{0.000000, 0.000000, 0.000000}
		\pgfsetstrokecolor{dialinecolor}
		\pgfsetstrokeopacity{1.000000}
		\pgfpathmoveto{\pgfpoint{18.000000\du}{8.000000\du}}
		\pgfpathcurveto{\pgfpoint{17.935688\du}{8.107187\du}}{\pgfpoint{17.764189\du}{8.150061\du}}{\pgfpoint{17.657003\du}{8.085749\du}}
		\pgfpathcurveto{\pgfpoint{17.549816\du}{8.021437\du}}{\pgfpoint{17.506942\du}{7.849939\du}}{\pgfpoint{17.571254\du}{7.742752\du}}
		\pgfpathcurveto{\pgfpoint{17.635566\du}{7.635566\du}}{\pgfpoint{17.807064\du}{7.592691\du}}{\pgfpoint{17.914251\du}{7.657003\du}}
		\pgfpathcurveto{\pgfpoint{18.021437\du}{7.721315\du}}{\pgfpoint{18.064312\du}{7.892813\du}}{\pgfpoint{18.000000\du}{8.000000\du}}
		\pgfpathclose
		\pgfusepath{fill,stroke}
		\pgfsetlinewidth{0.0500000\du}
		\pgfsetdash{}{0pt}
		\pgfsetbuttcap
		{
			\definecolor{diafillcolor}{rgb}{0.000000, 0.000000, 0.000000}
			\pgfsetfillcolor{diafillcolor}
			\pgfsetfillopacity{1.000000}
			% was here!!!
			\definecolor{dialinecolor}{rgb}{0.000000, 0.000000, 0.000000}
			\pgfsetstrokecolor{dialinecolor}
			\pgfsetstrokeopacity{1.000000}
			\draw (2.000000\du,2.000000\du)--(13.000000\du,5.000000\du);
		}
		\pgfsetlinewidth{0.0500000\du}
		\pgfsetdash{}{0pt}
		\pgfsetbuttcap
		{
			\definecolor{diafillcolor}{rgb}{0.000000, 0.000000, 0.000000}
			\pgfsetfillcolor{diafillcolor}
			\pgfsetfillopacity{1.000000}
			% was here!!!
			\definecolor{dialinecolor}{rgb}{0.000000, 0.000000, 0.000000}
			\pgfsetstrokecolor{dialinecolor}
			\pgfsetstrokeopacity{1.000000}
			\draw (2.000000\du,2.000000\du)--(2.000000\du,8.000000\du);
		}
		\pgfsetlinewidth{0.0500000\du}
		\pgfsetdash{}{0pt}
		\pgfsetbuttcap
		{
			\definecolor{diafillcolor}{rgb}{0.000000, 0.000000, 0.000000}
			\pgfsetfillcolor{diafillcolor}
			\pgfsetfillopacity{1.000000}
			% was here!!!
			\definecolor{dialinecolor}{rgb}{0.000000, 0.000000, 0.000000}
			\pgfsetstrokecolor{dialinecolor}
			\pgfsetstrokeopacity{1.000000}
			\draw (2.000000\du,8.000000\du)--(13.000000\du,5.000000\du);
		}
		\pgfsetlinewidth{0.0500000\du}
		\pgfsetdash{}{0pt}
		\pgfsetbuttcap
		{
			\definecolor{diafillcolor}{rgb}{0.000000, 0.000000, 0.000000}
			\pgfsetfillcolor{diafillcolor}
			\pgfsetfillopacity{1.000000}
			% was here!!!
			\definecolor{dialinecolor}{rgb}{0.000000, 0.000000, 0.000000}
			\pgfsetstrokecolor{dialinecolor}
			\pgfsetstrokeopacity{1.000000}
			\draw (7.000000\du,5.000000\du)--(18.000000\du,2.000000\du);
		}
		\pgfsetlinewidth{0.0500000\du}
		\pgfsetdash{}{0pt}
		\pgfsetbuttcap
		{
			\definecolor{diafillcolor}{rgb}{0.000000, 0.000000, 0.000000}
			\pgfsetfillcolor{diafillcolor}
			\pgfsetfillopacity{1.000000}
			% was here!!!
			\definecolor{dialinecolor}{rgb}{0.000000, 0.000000, 0.000000}
			\pgfsetstrokecolor{dialinecolor}
			\pgfsetstrokeopacity{1.000000}
			\draw (18.000000\du,2.000000\du)--(18.000000\du,8.000000\du);
		}
		\pgfsetlinewidth{0.0500000\du}
		\pgfsetdash{}{0pt}
		\pgfsetbuttcap
		{
			\definecolor{diafillcolor}{rgb}{0.000000, 0.000000, 0.000000}
			\pgfsetfillcolor{diafillcolor}
			\pgfsetfillopacity{1.000000}
			% was here!!!
			\definecolor{dialinecolor}{rgb}{0.000000, 0.000000, 0.000000}
			\pgfsetstrokecolor{dialinecolor}
			\pgfsetstrokeopacity{1.000000}
			\draw (7.000000\du,5.000000\du)--(18.000000\du,8.000000\du);
		}
		\node[anchor=west] at (-3\du,5\du){$T^2:$};
		\node[anchor=west] at (0\du,8\du){$1$};
		\node[anchor=west] at (0\du,2\du){$2$};
		\node[anchor=west] at (18\du,8\du){$6$};
		\node[anchor=west] at (18\du,2\du){$5$};
		\node[anchor=west] at (4.7\du,5\du){$3$};
		\node[anchor=west] at (13.5\du,5\du){$4$};
		\end{tikzpicture}
		
	\end{figure}
	
	Let $S=\kk[x_1,\ldots,x_6]$ the polynomial over a field $\kk$. The edge ideals of $T$ and $T^2$ are $$I(T)=\langle x_1x_3,x_2x_3,x_3x_4,x_4x_5,x_4x_6\rangle$$ and $$I\left(T^2\right)=\langle x_1x_3,x_2x_3,x_3x_4,x_4x_5,x_4x_6,x_1x_2,x_1x_4,x_2x_4,x_5x_6,x_3x_5,x_3x_6\rangle.$$ It is a simple matter to check that $I(T)$ has a linear resolution (it is easy to see that it has linear quotients), but $I(T^2)$ does not have a linear resolution. Combinatorially, this can be justified by the fact that the edges $\{1,2\}$ and $\{5,6\}$ form a gap in $T^2$, therefore in $(T^2)^c$ one can find the induced cycle $\{1,6,2,5\}$ with the vertices in this order. Therefore $(T^2)^c$ is not chordal and the result follows by Fr\"oberg's criterion.  
\end{Example}

In order to get the characterization, we start with several conditions that have to be fulfilled by trees such that the edge ideal of their square has a linear resolution. The first result determines the maximal diameter that the tree can have.
\begin{Lemma}\label{diam}
	Let $T$ be a tree and $G=T^2$. If $I(G)$ has a linear resolution, then $\diam(T)\leq4$.
\end{Lemma}
\begin{proof}Assume by contradiction that $\diam(T)\geq5$. Hence there exists an induced path of length $6$ in $T$, say $x,y,z,u,v,w$. Since $T$ is a tree, we may assume that $x$ is a free vertex. One may note that in $G$, the set of edges $\{x,y\}$ and $\{v,w\}$ form an induced gap. Indeed, this can be easily seen from the next figure and using the fact that $T$ is a tree, therefore between any two vertices there is a unique path. 
	\[\]
	\begin{center}
		%TeXCAD (http://texcad.sf.net/) Picture. File: [Ex13.pic]. Options on following lines.
		%\grade{\on}
		%\emlines{\off}
		%\epic{\off}
		%\beziermacro{\on}
		%\reduce{\on}
		%\snapping{\off}
		%\pvinsert{% Your \input, \def, etc. here}
		%\quality{8.000}
		%\graddiff{0.005}
		%\snapasp{1}
		%\zoom{4.0000}
		\unitlength 1mm % = 2.845pt
		\linethickness{0.4pt}
		\ifx\plotpoint\undefined\newsavebox{\plotpoint}\fi % GNUPLOT compatibility
		\begin{picture}(74.5,24.25)(0,0)
		\put(10.5,10){\line(1,0){47.5}}
		%\emline(10.25,20)(23.25,31.5)
		\multiput(10.25,10)(.0381231672,.0337243402){341}{\line(1,0){.0381231672}}
		%\end
		%\emline(57.75,20)(73,31)
		\multiput(57.75,10)(.0466360856,.0336391437){327}{\line(1,0){.0466360856}}
		%\end
		\put(23,21){\line(1,0){49.75}}
		%\emline(22.75,31.25)(32,20)
		\multiput(22.75,21.25)(.0336363636,-.0409090909){275}{\line(0,-1){.0409090909}}
		%\end
		%\emline(32,20)(48.25,31)
		\multiput(32,10)(.0496941896,.0336391437){327}{\line(1,0){.0496941896}}
		%\end
		%\emline(48.25,31)(58.25,20.25)
		\multiput(48.25,21)(.0336700337,-.0361952862){297}{\line(0,-1){.0361952862}}
		%\end
		\put(9,5){$x$}
		\put(31,5){$z$}
		\put(57,5){$v$}
		\put(22,23.25){$y$}
		\put(48,23.25){$u$}
		\put(73,23.25){$w$}
		%\put(3,26.75){$L_6^2:$}
		\end{picture}
	\end{center}
	Also note that any neighbour of $z$ in $T$, excepting $u$, is at distance at least 3 by both $v$ and $w$ and
	any neighbour of $u$, excepting $z$, is at distance at least $3$ by both $x$ and $y$. So $\{x,y\}$ and $\{v,w\}$ form an induced gap in $G$. This implies that $I(G)$ does not have a linear resolution, a contradiction.	\end{proof}
We determine now some restrictions on the degrees of the vertices.
\begin{Lemma}\label{claim1}
	Let $T$ be a tree with $\diam(T)=3$ such that $I(T^2)$ has a linear resolution. Then $T$ can contain at most one vertex of degree at least $3$.
\end{Lemma}
\begin{proof} Since $I(T^2)$ has a linear resolution, by Theorem~\ref{Froberg}, $(T^2)^c$ is a chordal graph, in particular $T$ is gap-free. Since $\diam(T)=3$, any two cut-points are adjacent. By Proposition~\ref{propT2}(d) the neighbourhoods of any two cut-points form cliques that are intersecting in exactly two vertices. 
	
	Let $v$ be a vertex of $T$, with $\deg_T(v)\geq 3$. Assume by contradiction that there exists a vertex $u\neq v$ such that $\deg_T(u)\geq3$. Hence there are $\{v,u_1,u_2\}\subseteq\mathcal{N}_T(u)$ and $\{u,v_1,v_2\}\subseteq\mathcal{N}_T(v)$ and, in $T^2$, $V(C(u))\cap V(C(v))=\{u,v\}$. Then $\{u_1,u_2\}$ and $\{v_1,v_2\}$ form an induced gap in $T^2$ (we are in the same situation as the one from Example~\ref{example}), a contradiction.
\end{proof}
We may now state the characterization.
\begin{Theorem}\label{linrestrees}
	Let $T$ be a tree and $G=T^2$ its square. The following are equi\-valent:
	\begin{itemize}
		\item[a)] $I(G)$ has a linear resolution.
		\item[b)] $T$ is one of the following graphs:\begin{itemize}
			\item[i)] $T$ is $L_n$ with $2\leq n\leq 5$;
			\item[ii)] $T$ is a star graph;
			\item[iii)] $T$ is a (partially) whiskered star.
		\end{itemize}  
	\end{itemize}
\end{Theorem}

\begin{proof} It is clear that, for every graph from b), $I(G)$  has a linear resolution. We assume that $I(G)$ has a linear resolution, therefore, by Lemma~\ref{diam}, one has that $\diam(T)\leq4$. If $n=2$, then $T$ is $L_2$. Therefore, we may assume that $n\geq 3$.
	
	According to Remark~\ref{gapfree}, $G$ is gap-free.
	
	If $G$ is a complete graph, then taking into account that $G=T^2$ and using Proposition~\ref{propT2}(a), one obtains that $T$ has to be a star graph. Note that, if $n=3$, we get $T=L_3$.
	
	Let's assume now that $G$ is not a complete graph. Therefore, there exist at least two cliques which implies that there are at least two cut-points in $T$. Moreover, one should note that $\diam(T)\geq3$. By Lemma~\ref{diam}, one has $\diam (T)=3$ or $\diam(T)=4$. We split the proof in two cases.
	
	\textbf{Case 1:} Assume that $\diam(T)=3$, therefore any two cut-points are adjacent. By Proposition~\ref{propT2}(d) every pair of cliques are intersecting in exactly two vertices. If all the vertices of $T$ have degree at most $2$, then $T$ is the graph $L_4$. If there is a vertex $v$ with $\deg_T(v)\geq3$, by Lemma~\ref{claim1} this vertex is unique. Such a tree is a whiskered star.
	
	\textbf{Case 2:} Assume that $\diam(T)=4$, in particular there is an induced path in $T$ with vertices $u,v,w,x,y$ such that $u$ and $y$ are free vertices in $T$. If $\deg_T(v)=\deg_T(w)=\deg_T(x)=2$, then $T$ is $L_5$. Assume that one of the vertices $v,w,x$ has degree at least $3$. One may verify that it can be only one (the proof is similar to the one for Lemma~\ref{claim1}). If we assume that $\deg_T(v)\geq 3$, then there is a vertex $v_1\in\mathcal{N}_T[v]\setminus\{u,w\}$. But $\{u,v_1\}$ and $\{x,y\}$ form an induced gap in $G$, a contradiction. By using similar arguments, one also gets that $\deg_T(x)=2$ Therefore, the only vertex that can have degree at least $3$ is $w$. Hence, $T$ is a whiskered star. 
\end{proof}
As a consequence, we get a complete characterization of trees $T$ with the property that the complement of their square is chordal, (that is $T^2$ is co-chordal):
\begin{Corollary}
	Let $T$ be a tree and $G=T^2$ its square. Then $G$ is a co-chordal graph if and only if $T$ is one of the following graphs:\begin{itemize}
		\item[i)] $T$ is $L_n$ with $2\leq n\leq 5$;
		\item[ii)] $T$ is a star graph;
		\item[iii)] $T$ is a (partially) whiskered star.
	\end{itemize}  
\end{Corollary}
\begin{proof} It follows easily by Theorem \ref{Froberg} and Theorem \ref{linrestrees}.
\end{proof}
The above results also lead to the following statement:
\begin{Corollary}
	If $T$ is a (partially) whiskered star graph, a star graph or $L_n$, with $2\leq n\leq 5$ and $G=T^2$, then:
	\begin{itemize}
		\item[a)] $\indmat (G^c)=1$
		\item[b)] $\cochord G=\cochord G^c=1.$
	\end{itemize}  
\end{Corollary}
\begin{proof}
	In order to prove $a)$, it is enough to see that $G^c$ is co-chordal since $G=T^2$ is chordal. Therefore $I(G^c)$ has a linear resolution by Theorem~\ref{Froberg}, thus $\reg\,S/I(G^c)=1=\indmat(G^c)$. The last equality follows by Proposition~\ref{indmatgen}b).
	
	For the second statement, one has to note that $I(G)$ has a linear resolution, therefore $G$ is co-chordal and $\cochord(G)=1$. Moreover, $G=T^2$ is chordal, therefore $G^c$ is co-chordal. Thus $\cochord(G^c)=1$.
\end{proof}
We close this section with the following remark. 
\begin{Remark}\rm
	In the view of Proposition~\ref{general}, the results obtained in this section are valid for larger classes of graphs which are not trees, but their square are isomorphic with the square of a tree.
\end{Remark}
\section{Edge ideals of squares of classes of trees} 
\label{sec:4}
In this section we will consider classes of trees such as paths and double brooms and we will determine some of their invariants such as the Krull dimension, the depth and the Castelnuovo--Mumford regularity. We recall some notions that will be intensively used in the sequel. We follow \cite{K} in order to fix the notations.

A graph $B$ is called \textit{a bouquet} if $B$ is a star graph with the vertex set $V(B)=\{w,z_1,\ldots,z_t\}$, $t\geq 1$, and $E(B)=\{\{w,z_i\}:1\leq i\leq t\}$ the set of edges. The vertex $w$ is called \textit{the root of} $B$, the vertices $z_1,\ldots,z_t$ are called \textit{the flowers of} $B$ and the edges of the graph $B$ are called \textit{stems}. We denote by $F(B)$ the set of flowers of $B$. Let $G$ be a graph with the vertex set $V(G)$, $E(G)$ be its set of edges, and $\mathcal{B}=\{B_1,\ldots,B_r\}$ a set of bouquets of $G$. Then \[\mathcal{F}(\mathcal{B})=\{z\in V(G):z\mbox{ is a flower in some bouquet from }\mathcal{B}\},\]
\[\mathcal{R}(\mathcal{B})=\{w\in V(G): \mbox{ is a root in some bouquet from }\mathcal{B}\}.\]

A set $\mathcal{B}$ of bouquets of $G$ is called \textit{semi-strongly disjoint} if $V(B_i)\cap V (B_j) = \emptyset$ for all $i\neq j$ and any two vertices belonging to $\mathcal{R}(\mathcal{B})$ are not adjacent in $G$.

Let $d'_G:=\max\{|\mathcal{F}(\mathcal{B})|: \mathcal{B} \mbox{ is a semi-strongly disjoint set of bouquets of } G\}$.

For the case of chordal graphs, the projective dimension can be computed in terms of semi-strongly disjoint sets. More precisely:
\begin{Theorem}\cite[Theorem 5.1]{K}\label{pdd'}
	Let $G$ be a chordal graph. Then $$\projdim S/I(G)=d'_G.$$
\end{Theorem} 

Firstly we consider the case of path graphs. Note that squares of paths have also been studied in \cite{II}, where an homological approach is used to determine the depth.

\begin{Proposition}\label{dimLn2} Let $n$ be an integer, $n\geq3$. Then $\dim S/I(L_n^2)=\left\lceil\frac{n}{3}\right\rceil$.
\end{Proposition}
\begin{proof}According to Proposition \ref{dimindep}, there is a maximal independent set $W$ such that $\dim S/I\left(L_n^2\right)=|W|$. Let's assume that $W$ is such a maximal independent set and $|W|=d$. We have to prove that $d=\lceil\frac{n}{3}\rceil$. 
	
	Since $E\left(L_n^2\right)=\{\{i,i+1\}:1\leq i\leq n-1\}\cup\{\{i,i+2\}:1\leq i\leq n-2\}$, one has that $|j-i|\geq 3$, for all $i,j\in W$. Therefore, $d\leq\lceil\frac{n}{3}\rceil$.
	
	For the other inequality, let's assume first that $n=3k$, that is $\lceil\frac{n}{3}\rceil=\left[\frac{n}{3}\right]=k$. Then the set $$\{1,4,7,\ldots,3k-2\}=\{1,1+3,\ldots,1+(k-1)3\}$$is also a maximal independent set of cardinality $k$. Hence $d\geq k=\lceil\frac{n}{3}\rceil$.
	
	If $n=3k+1$ or $n=3k+2$, then the set $$\{1,4,7,\ldots,3k-2,3k+1\}=\{1,1+3,\ldots,1+(k-1)3, 1+3k\}$$ is also a maximal independent set of cardinality $k+1=\lceil\frac{n}{3}\rceil$. Therefore we get $d\geq \lceil\frac{n}{3}\rceil$. The equality follows.
\end{proof}
Taking into account Theorem \ref{pdd'}, in order to determine the depth, we have to compute $d'_{L_n^2}$.
\begin{Proposition}\label{dg2} Let $n\geq 3$ be an integer. Then $d'_{L_n^2}=n-\left\lceil\frac{n}{5}\right\rceil$.
\end{Proposition}
\begin{proof} We split the proof in two cases.
	
	\textbf{{Case 1:}} $n\equiv k\mbox{ mod } 5$, $1\leq k\leq 4$.
	
	We denote $m=\left[\frac{n}{5}\right]$ and we consider the set of bouquets $\mathcal{B}=\{B_0,B_1,\ldots,B_m\}$ where \[\mathcal{R}(\mathcal{B})=\{1,6,11\ldots,5m+1\}\] and $$F(B_0)=\{2,3\}=\mathcal{N}_{L_n^2}(1),$$ $$F(B_i)=\mathcal{N}_{L_n^2}(5i+1)=\{5i-1,5i,5i+2,5i+3\},1\leq i\leq m-1$$and $$F(B_m)=\mathcal{N}_{L_n^2}(5m+1).$$
	Therefore, $\mathcal{F}(\mathcal{B})=V\setminus \mathcal{R}(\mathcal{B})$ and $| \mathcal{F}(\mathcal{B})|=n-\left\lceil\frac{n}{5}\right\rceil$.
	
	Moreover, one may note that $\mathcal{B}$ is a semi-strongly disjoint set of bouquets in $L_n^2$. Indeed, it is obvious that $V(B_i)\cap V(B_j)=\emptyset$, for all $i\neq j$, $0\leq i,j\leq m$. Since $j-i\geq5$ for all $i,j\in\mathcal{R}(\mathcal{B})$ with $i\neq j$, the set $\mathcal{R}(\mathcal{B})$ does not contain adjacent vertices in $L_n^2$. So both conditions are fulfilled.
	
	By the definition of $d'_{L_n^2}$, one must have $d'_{L_n^2}\geq n-\left\lceil\frac{n}{5}\right\rceil$. 
	
	In order to prove that we have equality, we remark that $F(\mathcal{B})=V\setminus R(\mathcal{B})$ implies that by increasing the number of roots, the number of flowers will decrease. Moreover, if one decreases the number of roots, even if one would consider the set of flowers as given by all the neighbours, the number of flowers will be strictly lower than the one we obtained. Therefore $d'_{L_n^2}=n-\left\lceil\frac{n}{5}\right\rceil$.
	
	\textbf{{Case 2:}} $n\equiv 0\mbox{ mod } 5$. We proceed as in the above case.
	We denote $m=\left[\frac{n}{5}\right]$ and we consider the set of bouquets $\mathcal{B}=\{B_0,B_1,\ldots,B_m\}$ where \[\mathcal{R}(\mathcal{B})=\{1,6,11\ldots,5m+1\}\] and $$F(B_0)=\{2,3\}=\mathcal{N}_{L_n^2}(1),$$ $$F(B_i)=\mathcal{N}_{L_n^2}(5i+1)=\{5i-1,5i,5i+2,5i+3\},1\leq i\leq m-1$$and $$F(B_m)=\{5m-1\}.$$
	Therefore, $\mathcal{F}(\mathcal{B})=V\setminus \mathcal{R}(\mathcal{B})$ and $|\mathcal{F}(\mathcal{B})|=n-\left\lceil\frac{n}{5}\right\rceil$. One may note that $\mathcal{B}$ is a semi-strongly disjoint set of bouquets in $L_n^2$.
	Arguing as before, one obtains $d'_{L_n^2}=n-\left\lceil\frac{n}{5}\right\rceil$. 
\end{proof}
By using the above result we can easily determine the depth  of $S/I\left(L_n^2\right)$ and the big height of $I\left(L_n^2\right)$. Note that the next result was also obtained in \cite[Theorem 3.8]{II} in a more general case.
\begin{Corollary}Let $n\geq3$ be an integer. Then $\depth\, S/I(L_n^2)=\left\lceil\frac{n}{5}\right\rceil$.
	\end{Corollary}
	\begin{proof} Since $L_n^2$ is a chordal graph, $\projdim\,S/I(L_n^2)=d'_{L_n^2}$, by Theorem~\ref{pdd'}. According to Proposition~\ref{dg2}, we have that $\projdim\S/I(L_n^2)=n-\left\lceil\frac{n}{5}\right\rceil$, therefore $\depth\,S/I(L_n^2)=\left\lceil\frac{n}{5}\right\rceil$.
\end{proof}
\begin{Corollary} Let $n\geq 3$ be an integer. Then $\bight(S/I(L_n^2))=n-\left\lceil\frac{n}{5}\right\rceil$.
\end{Corollary}
\begin{proof} By \cite[Corollary 5.6]{K}, $\bight(S/I(L_n^2))=\projdim(S/I(L_n^2))$. The statement follows.
\end{proof}

In order to compute the Castelnuovo--Mumford regularity, we have to determine the induced matching number of $L_n^2$.
\begin{Proposition}\label{indmatpath} Let $n\geq 3$ be an integer. Then $\indmat(L_n^2)=\left\lceil\frac{n-1}{4}\right\rceil$.
\end{Proposition}
\begin{proof} Let $k=\left\lceil\frac{n-1}{4}\right\rceil$, that is $n-1=4k-r$, where $0\leq r\leq 3$. We consider the set $$F=\{\{1,2\},\{5,6\},\ldots,\{4k-3,4k-2\}\},$$ where $4k-3=4(k-1)+1$. It is easily seen that $F$ is an induced matching, not necessarily maximal. Therefore $$\indmat(L_n^2)\geq|F|=\left\lceil\frac{n-1}{4}\right\rceil.$$
	
	For the other inequality, we consider an arbitrary maximal induced matching $F=\{\{i_1,j_1\},\{i_2,j_2\},\ldots,\{i_d,j_d\}\}$. Since $F$ is an induced matching, the inequalities $i_k-j_{k-1}\geq 3$ and $j_k\geq i_k+1$ should hold, for any $2\leq k\leq n$. Therefore $j_{k}-j_{k-1}\geq 4$ for any $2\leq k\leq n$, that is $|F|\leq\left\lceil\frac{n-1}{4}\right\rceil$.
	
\end{proof}

\begin{Theorem}\label{reg} Let $n\geq 3$ be an integer. Then $\reg(S/I(L_n^2))=\left\lceil\frac{n-1}{4}\right\rceil$.
\end{Theorem}
\begin{proof} According to Proposition~\ref{indmatgen}b), one has that $$\reg(S/I(L_n^2))=\indmat(L_n^2).$$ The statement follows by Proposition \ref{indmatpath}.
\end{proof}

In particular, we recover Theorem~\ref{linrestrees} case b(i) where we characterized all the squares of path graph whose edge ideal has a linear resolution:

\begin{Corollary} Let $n\geq 3$ be an integer. Then $S/I(L_n^2)$ has a linear resolution if and only if $n\leq 5$.
\end{Corollary}	
\begin{proof} $S/I(L_n^2)$ has a linear resolution if and only if $\reg(S/I(L_n^2))=1$, that is $\left\lceil\frac{n-1}{4}\right\rceil=1$ which is equivalent to $n\leq 5$.
\end{proof}

We consider now another particular class of trees called double brooms. In graph theory, \textit{a double broom} is a graph on $n_1+n+n_2$ vertices obtained from the path graph $L_n$ by appending to the first and the last vertex a set of $n_1$ and $n_2$ edges, respectively. One also denotes this double broom by $P(n_1,n,n_2)$. If $n=2$ the graph is called \textit{a double star}.
\begin{center}
	
	\begin{figure}[h]
		\includegraphics[height=2.3cm]{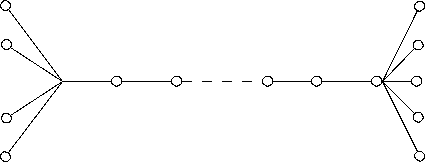}
		
		\caption{A double broom}
	\end{figure}
	
\end{center}

\newpage
We determine the projective dimension of $S/I(T^2)$. In order to do this, we will consider first several particular cases determined by the length of $L_n$. We begin with the case of double stars.
\begin{Proposition}
	Let $n_1,n_2\geq2$ be two integers and $T$ the double star on the set of vertices $V=\{x_1,\ldots,x_{n_1-1},y_{1},\ldots,y_{n_2-1},x,y\}$ and with the set of edges $$E=\{\{x,y\}\}\cup\{\{x,x_i\}:1\leq i\leq n_1-1\}\cup\{\{y,y_i\}:1\leq i\leq n_2-1\}.$$ Then $$\projdim S/I\left(T^2\right)=n_1+n_2-1.$$
\end{Proposition}
\begin{proof} It is easy to see that we may consider a bouquet $B$ in $T^2$ as follows: let $x$ be the root and $x_1,\ldots,x_{n_1-1},y,y_1,\ldots y_{n_2-1}$ the flowers. The stems of the bouquet are the edges which connect the root with each flower. Note that there are also edges of the form $\{x,y_i\}$ since $\dist_T(x,y_i)=2$ for all $1\leq i\leq n_2-1$. Since this is the maximal number of flowers that we can get, we have that $d'_T=n_1+n_2-1$. Since $T^2$ is chordal, we also get that $\projdim S/I(T^2)=n_1+n_2-1$.
\end{proof}
Similarly, we can prove the following result: 

\begin{Proposition} Let $T$ be the double broom $P(n_1-1,3,n_2-1)$ with the set of vertices $V=\{x_1,\ldots,x_{n_1-1},y_{1},\ldots,y_{n_2-1},x,y,z\}$ and with the set of edges $$E=\{\{x,z\},\{z,y\}\}\cup\{\{x,x_i\}:1\leq i\leq n_1-1\}\cup\{\{y,y_i\}:1\leq i\leq n_2-1\}.$$ Then $\projdim S/I\left(T^2\right)=n_1+n_2$. 
\end{Proposition}
\begin{proof} We prove as before. We consider the bouquet $B$ in $T^2$ as follows: let $z$ be the root and $x_1,\ldots,x_{n_1-1},x,y,y_1,\ldots y_{n_2-1}$ the flowers. The stems of the bouquet are the edges which connect the root with each flower (all these edges exist since the corresponding vertices are at distance at most $2$ in $T$). Since this is the maximal number of flowers that we can get, we have that $d'_T=n_1+n_2$. Since $T^2$ is chordal, we also get that $\projdim S/I(T^2)=n_1+n_2$.
\end{proof}
\begin{Proposition} Let $T$ be the double broom $P(n_1-1,k,n_2-1)$, $3<k\leq8$ with the set of vertices $V=\{x_1,\ldots,x_{n_1-1},y_{1},\ldots,y_{n_2-1},x,y,z_1,\ldots,z_{k-2},\}$, $|V|=n$,  and with the set of edges $$E=\{\{z_i,z_{i+1}:1\leq i\leq k-3 \}\}\cup\{\{x,z_1\},\{z_{k-2},y\}\}\cup\{\{x,x_i\}:1\leq i\leq n_1-1\}\cup$$ $$\cup\{\{y,y_i\}:1\leq i\leq n_2-1\}.$$ Then $\projdim S/I\left(T^2\right)=n-2$. 
\end{Proposition}
\begin{proof} We will consider two bouquets $B_1$ and $B_2$ in $T^2$ and their construction depends on $k$.
	
	\textbf{Case 1:} If $k=4$, let $x$ and $y$ be the roots, The flower set of $B_1$ is $\{x_1,\ldots,x_{n_1-1},z_{1}\}$ and the flower set of $B_2$ is $\{y_1,\ldots,y_{n_2-1},z_{2}\}$. 
	
	\textbf{Case 2:} If $k=5$, let $x$ and $y$ be the roots. The flower sets of $B_1$ and $B_2$ are $\{x_1,\ldots,x_{n_1-1},$ $z_{1},z_2\}$ and $\{y_1,\ldots,y_{n_2-1},z_{3}\}$ respectively. 
	
	\textbf{Case 3:} If $6\leq k\leq 8$, let $z_1$ and $z_{k-2}$ be roots and let $i=\left[\frac{k-2}{2}\right]$. Note that $i\leq 3$. We consider the flower set of $B_1$ to be $\{x_1,\ldots,x_{n_1-1},x,z_{2},\ldots z_{i}\}$, and the flower set of $B_2$ to be $\{y_1,\ldots,y_{n_2-1},y,z_{i+1},\ldots, z_{k-3}\}$. 
	
	Note that in each of the above cases the stems of the bouquet are the edges which connect the root with each flower from the corresponding bouquet (all these edges exist since the corresponding vertices are at distance at most $2$ in $T$). It is clear that the roots are not adjacent. Moreover, $V(G)=\mathcal{F}(\mathcal{B})\cup \mathcal{R}(\mathcal{B})$ where $\mathcal{B}=\{B_1,B_2\}$. Since this is the maximal number of flowers that we can get, we have that $d'_{T^2}=n-2$. Since $T^2$ is chordal, we also get that $\projdim S/I(T^2)=n-2$.
\end{proof}

We assume now that $k>8$ and we get the following result:
\begin{Proposition} Let $T$ be the double broom $P(n_1-1,k,n_2-1)$, $8<k$ with the set of vertices $V=\{x_1,\ldots,x_{n_1-1},y_{1},\ldots,y_{n_2-1},x,y,z_1,\ldots,z_{k-2},\}$, $|V|=n$,  and with the set of edges $$E=\{\{z_i,z_{i+1}:1\leq i\leq k-3 \}\}\cup\{\{x,z_1\},\{z_{k-2},y\}\}\cup\{\{x,x_i\}:1\leq i\leq n_1-1\}\cup$$ $$\cup\{\{y,y_i\}:1\leq i\leq n_2-1\}.$$ Then $\projdim S/I\left(T^2\right)=n-2-\left\lceil\frac{k-8}{5}\right\rceil$. 
\end{Proposition}
\begin{proof}We consider the following induced subgraphs of $T^2$: $G_1$ and $G_2$ are the induced subgraphs on the vertex set $\{x_1,\ldots,x_{n_1-1},x,z_1,z_2,z_3\}$ and $\{y_1,\ldots,y_{n_2-1},y,$ $z_{k-2},z_{k-3},z_{k-4}\}$, respectively, and $L^2_{k-8}$ the square of the path on the vertex set $\{z_4,\ldots,z_{k-5}\}$. It is a simple matter to check that $G_1=(T_{\mathcal{N}[z_1]})^2$ and $G_2=(T_{\mathcal{N}[z_{k-2}]})^2$.  We will show that $$d'_{T^2}=n_1+n_2+k-4-\left\lceil\frac{k-8}{5}\right\rceil.$$ Since $n=n_1+n_2+k-2$, the statement will follow. 
	
	We define a set of semi-strongly set of bouquets of $T^2$ as follows:
	\begin{itemize}
		\item[$-$] $B_1$ is the bouquet with the root $z_1$ and the flowers $\{x_1,\ldots,x_{n_1-1},x,z_2,z_3\}$. Note that the stems are $\{z_1,x_i\}$, $1\leq i\leq n_1-1$, $\{z_1,x\}$, $\{z_1,z_2\}$, $\{z_1,z_3\}$.
		\item[$-$] $B_2$ is the bouquet with the root $z_{k-2}$ and the flowers $\{y_1,\ldots,y_{n_2-1},y,z_{k-3},$ $z_{k-4}\}$. In this case, the stems are $\{z_{k-2},y_i\}$, $1\leq i\leq n_2-1$, $\{z_{k-2},y\}$, $\{z_{k-2},z_{k-3}\}$, $\{z_{k-2},z_{k-4}\}$.
		\item[$-$] $\mathcal{B'}$ is the strongly disjoint set of bouquets of the graph $L^2_{k-8}$ defined in Proposition \ref{dg2}.
	\end{itemize}
	Defined like this, $\mathcal{B}=\{B_1,B_2\}\cup\mathcal{B'}$ is a strongly disjoint set of bouquets of $T^2$. In particular, $d'_{T^2}\geq n_1+2+n_2+2+d'_{L^2_{k-8}}=n_1+n_2+k-4-\left\lceil\frac{k-8}{5}\right\rceil$. 
	
	We see at once that $\deg_{T^2}(z_1)=n_1+2$, $\deg_{T^2}(z_{k-2})=n_2+2$ and the degree of every root of $\mathcal{B}'$ is at most $4$. Therefore, we obtained the maximal number of flowers. This implies that $$d'_{T^2}= n_1+2+n_2+2+d'_{L^2_{k-8}}=n-2-\left\lceil\frac{k-8}{5}\right\rceil.$$ Since $T^2$ is chordal, the statement follows.
\end{proof}

We can summarize the above results as follows:
\begin{Theorem}
	Let $T$ be the double broom $P(n_1-1,k,n_2-1)$, with $n_1,n_2,k\geq2$ and $n=n_1+n_2+k-2$. Then$$\projdim(S/I(T^2))=\left\{\begin{array}{cc}
	n-1,&\mbox{ if } k\in\{2,3\}\\
	n-2,&\mbox{ if } 3<k\leq8\\
	n-2-\left\lceil\frac{k-8}{5}\right\rceil,&\mbox{ if } k>8
	\end{array}\right..$$
\end{Theorem}
As a consequence, we get 
\begin{Corollary}\label{depth}
	Let $T$ be the double broom $P(n_1-1,k,n_2-1)$, with $n_1,n_2,k\geq2$ and $n=n_1+n_2+k-2$. Then$$\depth(S/I(T^2))=\left\{\begin{array}{cc}
	1,&\mbox{ if } k\in\{2,3\}\\
	2,&\mbox{ if } 3<k\leq8\\
	2+\left\lceil\frac{k-8}{5}\right\rceil,&\mbox{ if } k>8
	\end{array}\right..$$
\end{Corollary}
Next, we determine the Krull dimension for edge ideals of squares double brooms.

\begin{Proposition}
	Let $T=P(n_1-1,k,n_2-1)$ where $n_1,n_2\geq2$ and $k\geq4$. Then $$\dim S/I(T^2)=\left\lceil\frac{k-4}{3}\right\rceil+2.$$
\end{Proposition}
\begin{proof} In order to fix the notations, let's assume that the set of vertices of $T$ is $V=\{x_1,\ldots,x_{n_1-1},y_{1},\ldots,y_{n_2-1},x,y,z_1,\ldots,z_{k-2},\}$, $|V|=n$,  and with the set of edges $$E=\{\{z_i,z_{i+1}:1\leq i\leq k-3 \}\}\cup\{\{x,z_1\},\{z_{k-2},y\}\}\cup\{\{x,x_i\}:1\leq i\leq n_1-1\}\cup$$ $$\cup\{\{y,y_i\}:1\leq i\leq n_2-1\}.$$ 
	
	According to Proposition \ref{dimindep}, there is a maximal independent set $W$ such that $\dim S/I\left(T^2\right)=|W|$. Let's assume that $W$ is such a maximal independent set and $|W|=d$. We have to prove that $d=\lceil\frac{k-4}{3}\rceil+2$. 
	
	If $k=4$, we consider the induced subgraphs $G_1=C(x)=\mathcal{K}_{n_1+1}$ and $G_2=C(y)=\mathcal{K}_{n_2+1}$. Then any maximal independent set can have at most two vertices. Since $W=\{x,y\}$ is a maximal independent set, $d=2$.
	
	If $k=5$, then $\diam(T^2)=4$. We consider as before the induced subgraphs $G_1=C(x)=\mathcal{K}_{n_1+1}$ and $G_2=C(y)=\mathcal{K}_{n_2+1}$. Note that $\{x_i,z_2\}\notin E(T^2)$ and $\{y_j,z_2\}\notin E(T^2)$ since the distance in $T$ between these vertices is $3$ for all $1\leq i\leq n_1-1$ and $1\leq j\leq n_2-1$. Then $W=\{x_1,z_2,y_1\}$ is a maximal independent set and, taking into account the shape of the graph, it has maximal cardinality, so $d=3$.
	
	If $k>5$, then we consider the induced subgraphs $G_1=C(x)=\mathcal{K}_{n_1+1}$, $G_2=C(y)=\mathcal{K}_{n_2+1}$ and $L_{k-4}^2$ which is the path on the vertices $z_2,\ldots, z_{k-3}$. In order to obtain a maximal independent set of maximal cardinality, one has to take a vertex from $G_1$, a vertex from $G_2$ and a maximal independent set of maximal cardinality for $L_{k-4}^2$. Since the largest maximal independent set of $L_{k-4}^2$ has $\left\lceil\frac{k-4}{3}\right\rceil$ (by Proposition~\ref{dimLn2}), the statement follows. 
\end{proof}

\section{Open questions and remarks}
\label{sec:5}
We end this paper with several remarks and open questions. The starting point of  this paper was to consider the behaviour of the invariants of the edge ideal when one consider the square of the graph. This was suggested by the fact that in Combinatorics, many researchers paid attention to combinatorial properties that are preserved by the square \cite{CC,LT,LT1,L,R,R1}. From the commutative algebra point of view, examples show that, there are large classes of trees for which the Castelnuovo--Mumford regularity of the edge ideal of the square decreases. In fact, the tree from Example~\ref{example} is the smallest one that we could find for which the regularity increases. Therefore, the following problem naturally appears:
\begin{Problem}
	Characterize all trees $T$ for which $\reg\,I(T))\geq\reg\,I(T^2)$.
\end{Problem}
One can also consider the behaviour of the projective dimension. Note that for path graphs, Morey proved that $\depth\,S/I(L_n)=\left\lceil\frac{n}{3}\right\rceil$ (\cite[Lemma 2.8]{M}) and we showed that $\depth\,S/I(L^2_n)=\left\lceil\frac{n}{5}\right\rceil$. Therefore $\projdim\,I(L_n)\leq\projdim\,I(L_n^2)$. Examples suggest that this is true in general. Therefore we assume that the next question has a positive answer:
\begin{Question}
	Is it true that if $T$ is a tree then $\projdim\,I(T)\leq\projdim\, I(T^2)$?
\end{Question}
If the above question has a negative answer, then one can consider the following problem:
\begin{Problem}
	Characterize all trees $T$ for which $$\projdim\,I(T)\leq\projdim\,I(T^2).$$
\end{Problem}
Note that same questions can be considered for different classes of graphs.

\end{document}